\documentclass[11pt,leqno,twoside]{amsart}

\usepackage{fancyvrb}
\usepackage{fancyhdr}
\usepackage{etoolbox}
\usepackage{enumerate}
\usepackage{pgfplots}
\pgfplotsset{compat=1.16, width = 10cm}

\usepackage{amssymb}
\usepackage{placeins}
\usepackage[english]{babel}
\usepackage{amssymb,amsthm,amsmath,eucal,mathrsfs}
\usepackage{bm}

\usepackage{amsmath}
\usepackage{amsfonts}
\usepackage{float}
\usepackage{amsmath}
\usepackage{amssymb}
\usepackage{graphicx}
\usepackage{lscape}
\usepackage{amstext}
\usepackage{amsthm}
\usepackage{color}
\usepackage{float}
\usepackage{mathrsfs}
\usepackage{epsfig}
\usepackage{url}
\usepackage{fancyhdr}
\usepackage{pspicture}
\usepackage{graphicx}

\usepackage{cite}

\usepackage{amsmath,verbatim}

\usepackage{amsthm}

\usepackage{amssymb}

\usepackage{amsfonts}

\usepackage{dsfont}

\usepackage{hyperref}

\newtheorem{theorem}{Theorem}[section]

\newtheorem{proposition}[theorem]{Proposition}

\newtheorem{Hypotheses}{Hypotheses}

\theoremstyle{definition}

\newtheorem{remark}[theorem]{Remark}

\numberwithin{equation}{section}

\newcommand\restr[2]{{
  \left.\kern-\nulldelimiterspace 
  #1 
  \vphantom{\big|} 
  \right|_{#2} 
  }}

\usepackage{eucal}

\date{\today}

\title[Logarithmic potential]{Estimation of the eigenvalues and the integral of the eigenfunctions of the Newtonian potential operator}

\author[Alsenafi,, Ghandriche and Sini]{Abdulaziz Alsenafi$^{*}$, Ahcene Ghandriche  $^{**}$ and Mourad Sini$^{\ddag}$}
\thanks{$^*$ Department of Mathematics, Faculty of Science, Kuwait University, P.O. Box 5969, Safat 13060, Kuwait. Email: abdulaziz.alsenafi@ku.edu.kw}
\thanks{$^*$ Nanjing Center for Applied Mathematics, Nanjing 211135, People's Republic of China. Email: gh.hsen@njcam.org.cn.}
\thanks{$^{\ddag}$ RICAM, Austrian Academy of Sciences, Altenbergerstrasse 69, A-4040, Linz, Austria. Email: mourad.sini@oeaw.ac.at.} 

\begin{document}

\maketitle

\begin{abstract}
We consider the problem of estimating the  eigenvalues and the integral of the corresponding eigenfunctions, associated to the Newtonian potential operator, defined in a bounded domain $\Omega \subset \mathbb{R}^{d}$, where $d=2,3$, in terms of the maximum radius of $\Omega$. 
We first provide these estimations in the particular case of a ball and a disc. Then we extend them to general shapes using a, derived, monotonicity property of the eigenvalues of the Newtonian operator. The derivation of the lower bounds are quite tedious for the 2D-Logarithmic potential operator. Such upper/lower bounds appear naturally while estimating the electric/acoustic fields propagating in $\mathbb{R}^{d}$ in the presence of small scaled and highly heterogeneous particles.
\end{abstract}

\bigskip

\section{Introduction and statement of the results}
\subsection{Introduction}\label{SubsectionI}
The Newtonian potential integral operator, in $\mathbb{R}^{d}$, $d=2,3$, is of considerable interest in both scattering and potential theory, without being exhaustive we refer the readers to \cite{cartan1945theorie, ammari2019subwavelength, alsenafi2022foldy, ghandriche2022mathematical, colton2019inverse, ammari2018super}. 

Let $\Omega$ be a bounded and Lipschitz-regular domain of $\mathbb{R}^{d}$, $d=2,3$. The Newtonian operator correspond to any function $f$ the potential
\begin{equation}\label{defN}
u(x) = N_{\Omega}\left( f \right)(x) := \int_{\Omega} \phi_{0}(x,y) \, f(y) \, dy, \quad x \in \Omega,
\end{equation}
 where 
\begin{equation}\label{FundamentalSolution}
\phi_{0}(x,y) := \left\{
\begin{array}{rll}
  \dfrac{-1}{2 \, \pi} \log\left\vert x - y \right\vert \quad & \text{in} & \quad \mathbb{R}^{2},\\
 && \\
    \dfrac{1}{4 \, \pi \, \left\vert x-y \right\vert} \quad & \text{in} & \quad \mathbb{R}^{3}, 
    \end{array}
\right.
\end{equation}
is the fundamental solution of the Laplacian operator, i.e., 
\begin{equation*}
- \Delta \phi_{0}(x,y) = \delta(x,y), \quad \text{in} \quad \mathbb{R}^{d} \quad \text{for} \; d=2,3,
\end{equation*}
where $\delta(\cdot,\cdot)$ is the Dirac function. It is well known that $N_{\Omega}(\cdot)$, 
is a linear, compact, self-adjoint non-negative operator on $\mathbb{L}^{2}(\Omega)$ and it carries $\mathbb{L}^{2}(\Omega)$ to $\mathbb{H}^{2}(\Omega)$. For other useful properties  of Newtonian potential operator we refer the readers to \cite{Anderson, kellogg, ruzhansky2016isoperimetric, kalmenov2011boundary}. Therefore, $N_{\Omega}(\cdot)$ has a countable decreasing sequence of eigenvalues, with possible multiplicities, that we denote in the sequel by $\left\{ \lambda_{n}(\Omega) \right\}_{n \geq 0}$, with the corresponding eigenfunctions as a basis of the space $\mathbb{L}^{2}(\Omega)$. Unfortunately, the literature on computing explicitly its eigenvalues and the corresponding eigenfunctions is not that rich and it has focused only on 'simple' domains, e.g. symmetric domains, see \cite{kalmenov2011boundary, Anderson}.  For the particular case of a ball, in $\mathbb{R}^{3}$, of radius $a$, the authors of \cite[Theorem 4.2]{kalmenov2011boundary} have given explicit expressions for the eigenvalues defined by 
\begin{equation}\label{EigValFractionalOrder}
\lambda_{l,j} = \frac{a^{2}}{\left[ \mu_{j}^{(l + \frac{1}{2})} \right]^{2}}, \quad l \geq 0 \quad \text{and} \quad j \geq 1, 
\end{equation}
where $\mu_{j}^{(l + \frac{1}{2})}$ are the roots of the transcendental equation
\begin{equation}\label{muFractionalOrder}
\left(2 \, l + 1 \right) \; \begin{LARGE}
\textbf{J}_{l + \frac{1}{2}}
\end{LARGE}\left( \mu_{j}^{(l + \frac{1}{2})} \right) + \frac{\mu_{j}^{(l + \frac{1}{2})}}{2} \; \left(     \begin{LARGE}
\textbf{J}_{l - \frac{1}{2}}
\end{LARGE}\left( \mu_{j}^{(l + \frac{1}{2})}\right)   -  \begin{LARGE}
\textbf{J}_{l + \frac{3}{2}}
\end{LARGE}\left( \mu_{j}^{(l + \frac{1}{2})} \right) \right) = 0,
\end{equation}
where $\begin{LARGE}
\textbf{J}_{\nu}
\end{LARGE}\left( \cdot \right)$, \, for $\nu \in \mathbb{R}$, refers to the Bessel function of fractional order. We recall that $\begin{LARGE}
\textbf{J}_{\nu}
\end{LARGE}\left( \cdot \right)$, admits the following representation 
\begin{equation}\label{BFFOS}
\begin{LARGE}
\textbf{J}_{\nu}
\end{LARGE}\left( x \right) = \sum_{k = 0}^{\infty} \frac{(-1)^{k}}{\Gamma(k+1) \, \Gamma(k+1+\nu)} \; \frac{x^{2k+\nu}}{2^{2k+\nu}},
\end{equation}
where $\Gamma(\cdot)$ stands for the Gamma function. The eigenfunctions corresponding to each eigenvalue $\lambda_{l,j}$ can be represented, in spherical coordinates, in the form
\begin{equation}\label{EigFctFO}
u_{l,j,m}(r, \phi, \theta) = \begin{LARGE}
\textbf{J}_{l + \frac{1}{2}}
\end{LARGE}\left(\sqrt{\lambda_{l,j}} \, r \right) \mathbb{Y}_{l}^{m}(\phi, \theta),
 \quad \text{with} \;\; \left\vert m \right\vert \leq l, 
\end{equation}
where 
\begin{equation}\label{SphericalHarmonics}
\mathbb{Y}_{l}^{m}(\phi, \theta) = \left\{
\begin{array}{rll}
  \mathbb{P}_{l}^{m}\left( \cos(\theta) \right) \; \cos(m \phi) \quad & \text{for} & \quad m=0,\cdots,l ,\\
 && \\
   \mathbb{P}_{l}^{\left\vert m \right\vert}\left( \cos(\theta) \right) \; \sin(\left\vert m \right\vert \phi) \quad & \text{for} & \quad m=-1,\cdots,-l , 
    \end{array}
\right.
\end{equation}
where $\mathbb{P}_{l}^{m}$ are the associated Legendre polynomials. It is clear, from $(\ref{muFractionalOrder})$, that $\mu_{j}^{(l + \frac{1}{2})} \sim 1$, for $l \geq 0$ and $j \geq 1$. Hence, from $(\ref{EigValFractionalOrder})$, we obtain
\begin{equation}\label{EVNB}
\lambda_{l,j} \, \sim \, a^{2}, \quad \text{for} \; l \geq 0 \; \text{and} \; j \geq 1. 
\end{equation}
For the integral of the eigenfunctions, we have 
\begin{equation*}
\int_{D} u_{l,j,m}(x) \, dx \overset{(\ref{EigFctFO})}{=} \int_{0}^{a} \, \begin{LARGE}
\textbf{J}_{l + \frac{1}{2}}
\end{LARGE}\left(\sqrt{\lambda_{l,j}} \, r \right) \,  r^{2} \, dr \, \int_{0}^{2 \, \pi} \, \int_{0}^{\pi}  \mathbb{Y}_{l}^{m}(\phi, \theta) \, \cos(\theta) \, d\theta \, d\phi.  
\end{equation*}
Observe, from $(\ref{SphericalHarmonics})$, that for $m \neq 0$
\begin{equation*}
\int_{0}^{2 \, \pi} \, \int_{0}^{\pi}  \mathbb{Y}_{l}^{m}(\phi, \theta) \, \cos(\theta) \, d\theta \, d\phi = 0. 
\end{equation*}
Then, by taking $m=0$ we obtain
\begin{equation*}
\int_{D} u_{l,j,0}(x) \, dx = 2 \, \pi \, \int_{0}^{a} \, \begin{LARGE}
\textbf{J}_{l + \frac{1}{2}}
\end{LARGE}\left(\sqrt{\lambda_{l,j}} \, r \right) \,  r^{2} \, dr \,\, \int_{0}^{\pi}  \mathbb{P}_{l}^{0}(\cos(\theta)) \, \cos(\theta) \, d\theta.  
\end{equation*}
We can check that\footnote{This can be proved using the series representation of $\mathbb{P}_{2l}^{0}(\cdot)$, given by
\begin{equation*}
\mathbb{P}_{2l}^{0}(x) = \frac{1}{4^{l}} \sum_{m=0}^{l} (-1)^{m} \frac{(4l-2m)!}{m! \, (2l-m)! \, (2l-2m)!} \, x^{2l-2m}.
\end{equation*}
} for $l$ even, we have
\begin{equation*}
\int_{0}^{\pi}  \mathbb{P}_{l}^{0}(\cos(\theta)) \, \cos(\theta) \, d\theta = 0.
\end{equation*}
Hence, by keeping only the odd index, we end up with
\begin{eqnarray*}
\int_{D} u_{2l+1,j,0}(x) \, dx &=& 2 \, \pi \, \int_{0}^{a} \, \begin{LARGE}
\textbf{J}_{2 l + \frac{3}{2}}
\end{LARGE}\left(\sqrt{\lambda_{2l+1,j}} \, r \right) \,  r^{2} \, dr \,\, \int_{0}^{\pi}  \mathbb{P}_{2l+1}^{0}(\cos(\theta)) \, \cos(\theta) \, d\theta \\ 
&\overset{(\ref{EigValFractionalOrder})}{=}& 2 \, \pi \, \left( \dfrac{\mu_{j}^{(2 l + \frac{3}{2})}}{a} \right)^{3} \, \int_{0}^{\dfrac{a^{2}}{\mu_{j}^{(2 l + \frac{3}{2})}}} \, \begin{LARGE}
\textbf{J}_{2 l + \frac{3}{2}}
\end{LARGE}\left( r \right) \,  r^{2} \, dr \,\, \int_{0}^{\pi}  \mathbb{P}_{2l+1}^{0}(\cos(\theta)) \, \cos(\theta) \, d\theta.  
\end{eqnarray*}
Using the series representation of $\mathbb{P}_{2l+1}^{0}(\cdot)$, we obtain: 
\begin{equation*}
\int_{0}^{\pi}  \mathbb{P}_{2l+1}^{0}(\cos(\theta)) \, \cos(\theta) \, d\theta = \sum_{m=0}^{l} (-1)^{m} \, \frac{\left( 4 l + 2 - 2m  \right)! \, \pi}{2^{2(2l+1-m)} \, (l-m)! \; (l+1-m)!} \; \sim \; 1.
\end{equation*}
Then, knowing that $\mu_{j}^{(2 l + \frac{3}{2})} \, \sim \, 1$, we deduce
\begin{equation}\label{ApprInt}
\int_{D} u_{2l+1,j,0}(x) \, dx  \sim  a^{-3}  \, \int_{0}^{a^{2}} \, \begin{LARGE}
\textbf{J}_{2 l + \frac{3}{2}}
\end{LARGE}\left( r \right) \,  r^{2} \, dr.  
\end{equation}
Now, using $(\ref{BFFOS})$,
\begin{equation*}
\int_{0}^{a^{2}} \, \begin{LARGE}
\textbf{J}_{2 l + \frac{3}{2}}
\end{LARGE}\left( r \right) \,  r^{2} \, dr = a^{9 + 4 l} \sum_{k = 0}^{\infty} \frac{(-1)^{k}}{\Gamma(k+1) \, \Gamma\left(k+2l+\frac{5}{2}\right)} \; \frac{1}{2^{2k+2l+\frac{3}{2}}} \; \frac{a^{4k}}{\frac{9}{2} + 2k + 2l} \; \sim \; a^{9 + 4 l}.
\end{equation*}
Hence, 
\begin{equation}\label{NU}
\int_{D} u_{2l+1,j,0}(x) \, dx  \, \sim \,  a^{6+4l}.  
\end{equation}
For the $\mathbb{L}^{2}(D)$ norm of $u_{2l+1,j,0}(\cdot)$, similarly to $(\ref{ApprInt})$, we can derive 
\begin{equation}\label{DU}
\left\Vert u_{2l+1,j,0} \right\Vert^{2}_{\mathbb{L}^{2}(D)} \; \sim \; a^{-3} \, \, \int_{0}^{a^{2}} \left\vert \begin{LARGE}
\textbf{J}_{2 l + \frac{3}{2}}
\end{LARGE}\left( r \right) \right\vert^{2} \,  r^{2} \, dr \; \overset{(\ref{BFFOS})}{\sim} \; a^{9 + 8 l}.  
\end{equation}
By setting $v_{2l+1,j,0} := \dfrac{u_{2l+1,j,0}}{\left\Vert u_{2l+1,j,0} \right\Vert_{\mathbb{L}^{2}(D)}}$ to be the normalized eigenfunctions and combining $(\ref{NU})$ with $(\ref{DU})$, we end up with the following estimation
\begin{equation}\label{EFNB}
\int_{D} v_{2l+1,j,0}(x) \, dx \; \sim \; a^{\frac{3}{2}}.
\end{equation} 
Finally, in the case of ball of radius $a$, we deduce from  $(\ref{EVNB})$ and $(\ref{EFNB})$, the following behaviour 
\begin{equation}\label{3DEigS}
\lambda_{l,j} \, \sim \, a^{2}  \quad \text{and} \quad \int_{D} v_{2l+1,j,0}(x) \, dx \; \sim \; a^{\frac{3}{2}} \quad \text{for } l \geq 0 \;\, \text{and} \;\, j \geq 1. 
\end{equation}
Deriving analogous formula for $(\ref{3DEigS})$, in the case of two dimension space is more complicated. The reason for this is the additive logarithmic term appearing after scaling the fundamental solution in two dimension. More precisely, from $(\ref{FundamentalSolution})$, we see that 
\begin{equation}\label{KernelScale}
\phi_{0}(x,y) = a^{-1} \, \phi_{0}\left( \tilde{x},\tilde{y} \right), \;\; \text{in 3D}, \;\; \text{and} \;\; \phi_{0}(x,y) = \frac{-1}{2 \, \pi} \log(a) + \phi_{0}\left( \tilde{x},\tilde{y} \right), \;\; \text{in 2D},
\end{equation}
where $x = z + a \, \tilde{x}$ and $y = z + a \, \tilde{y}$. The goal of this work is to analyse in detail, for an arbitrary domain $\Omega$ the scale of the eigenvalues and the integral of their corresponding eigenfunctions for the two dimensional Newtonian potential operator, that is also called in the literature the Logarithmic potential operator. As done previously, we start with the simplest case of a disc. For the case of a disc of radius $a$, we recall from \cite[Theorem 4.1]{kalmenov2011boundary} the explicit expressions of the eigenvalues 
\begin{equation*}
\lambda_{k,j} = a^{2} \, \left( \mu_{j}^{(k)} \right)^{-2}, \qquad k=0,1,2,\cdots \quad \text{and} \quad j=1,2,\cdots 
\end{equation*} 
and the corresponding eigenfunctions are given by
\begin{equation*}
u_{k,j}(r,\phi) = \begin{LARGE}
\textbf{J}_{k}
\end{LARGE}\left(\mu_{j}^{(k)} \, \frac{r}{a}\right) \, e^{i k \phi}, \qquad r \in [0,a] \quad \text{and} \quad \phi \in [0,2\pi],
\end{equation*} 
where $\begin{LARGE}
\textbf{J}_{k}
\end{LARGE}$ is the Bessel function of the first kind of order $k$ and $\mu_{j}^{(k)}$ are the roots of the following transcendental equation
\begin{equation*}
k \, \begin{LARGE}
\textbf{J}_{k}
\end{LARGE}\left( \mu_{j}^{(k)} \right) + \frac{\mu_{j}^{(k)}}{2} \, \left( \begin{LARGE}
\textbf{J}_{k-1}
\end{LARGE}\left( \mu_{j}^{(k)} \right) - \begin{LARGE}
\textbf{J}_{k+1}
\end{LARGE}\left( \mu_{j}^{(k)} \right) \right) = 0, \qquad k=1,2,\cdots,
\end{equation*}
and, for $k=0$, 
\begin{equation*}
\begin{LARGE}
\textbf{J}_{0}
\end{LARGE}
\left( \mu_{j}^{(0)} \right) + 2 \, \log(a) \, \mu_{j}^{(0)} \, \begin{LARGE}
\textbf{J}_{1}
\end{LARGE}
\left( \mu_{j}^{(0)} \right) = 0.
\end{equation*}
\medskip
\newline
Writing such explicit formulas for the eigenvalues and the eigenfunctions for an arbitrary domain $\Omega$ is out of reach. To overcome this difficulty, we propose a two-steps method allowing us to get the scale of the eigenvalues and the integral of eigenfunctions of the Newtonian potential operator defined over an arbitrary domain $\Omega$. First, we estimate the scale of the eigenvalues and the integral of eigenfunctions of Newtonian operator defined over a disc of radius $a$. Afterwards, the idea is to encircle the  domain $\Omega$, from inside and outside, between two discs, that we denote in the sequel by $D_{1}$ and $D_{2}$, with radius of each of them proportional to $a$ and then we make use of the \textit{property domain monotonicity} of the eigenvalues of the Newtonian potential operator that we prove in Section \ref{SectionII}, to derive the scale of the eigenvalues and eigenfunctions of the Newtonian potential operator. To accomplish this, the coming assumption, regarding the shape domain $\Omega$, is needed to derive the \textit{ property of domain monotonicity} for the Newtonian potential operator.
\begin{Hypotheses}\label{Hyp}
The domains $ \Omega $ are taken to be Lipschitz-regular domain of $\mathbb{R}^2$ and satisfy the following property 
\begin{equation}\label{EH1}
D(z_{1};\rho_{1}) := D_{1} \subset \Omega \subset D_{2} := D(z_{2};\rho_{2}),
\end{equation}
where, for $j=1,2$, we have 
\begin{equation}\label{EH2}
z_{j} \in \Omega \;\; \text{and} \;\; \rho_{j} \; \sim \; a,
\end{equation}
and $D_{j}$ is a disc.
\end{Hypotheses}
The previous hypotheses suggest that
\begin{equation*}
\left\vert \Omega \right\vert \;\; \sim \;\; \pi \, \rho_{j}^{2} \;\; \sim \;\; a^{2}.
\end{equation*}

\bigskip
\subsection{Statement of the results}
In the following theorem, we state the scales of the eigenvalues and the integral of eigenfunctions of the Newtonian potential operator, defined by $(\ref{defN})$, for an arbitrary domain $\Omega$ satisfying \textbf{Hypotheses \ref{Hyp}}.

\begin{theorem}\label{PrincipleResult}
Assume that $\Omega \subset \mathbb{R}^{2}$ satisfies \textbf{Hypotheses \ref{Hyp}}. Let\footnote{To write short formula we omit to mark the dependency of the eigenfunctions with respect to the domain $\Omega$.} $\left( \lambda_{n}(\Omega); e_{n} \right)_{n \geq 0}$ be the eigen-system associated to the 2D-Newtonian potential operator $N_{\Omega}(\cdot)$, defined by $(\ref{defN})$, then we have the following behavior
\begin{enumerate}
\item[]
\item For $n=0$, we have 
\begin{equation}\label{FirstEigValOmega}
\lambda_{0}\left( \Omega \right) \sim a^{2} \, \left\vert \log(a) \right\vert \quad \text{and} \quad \int_{\Omega} e_{0}(x) \, dx \sim a.
\end{equation}
\item[]
\item For $n \geq 1$, we have 
\begin{equation*}
\lambda_{n}\left( \Omega \right) \sim a^{2}  \quad \text{and} \quad \vert \int_{\Omega} e_{n}(x) \, dx \vert \lesssim a \, \left\vert \log(a) \right\vert^{- 1}.
\end{equation*}
\end{enumerate}
In the particular case when $\Omega$ is a disc of radius $a$, i.e. $\Omega = D$, we have the following behavior 
\begin{enumerate}
\item[]
\item For $n=0$, we have\footnote{Because the eigenvalues of the Newtonian operator are decreasing and as we see that $\lambda_{0,1}$ is the highest one, we refer to $\lambda_{0,1}$ to be the first eigenvalue.} 
\begin{equation}\label{FirstEigValDisc}
\lambda_{0,1}\left( D \right) \sim a^{2} \, \left\vert \log(a) \right\vert \quad \text{and} \quad \int_{D} e_{0,1}(x) \, dx \sim a.
\end{equation}
\begin{equation*}
\qquad \qquad \quad \lambda_{0,j}\left( D \right) \sim a^{2} \,  \quad \text{and} \quad \int_{D} e_{0,j}(x) \, dx \sim  \; a \; \left\vert \log(a) \right\vert^{-1}, \quad j \geq 2.
\end{equation*}
\item[]
\item For $n \geq 1$, we have 
\begin{equation}\label{VanishingInt}
\lambda_{n,j}\left( D \right) \sim a^{2}  \quad \text{and} \quad \int_{D} e_{n,j}(x) \, dx = 0, \quad j \geq 1.
\end{equation}
\end{enumerate}
\end{theorem}
\bigskip
We have seen in Subsection \ref{SubsectionI}, how the eigenvalues and the integral of their corresponding eigenfunctions scales with respect to the radius of the used ball, that we have denoted by $'a'$, see for instance $(\ref{3DEigS})$. Based on the previous theorem, the goal of the coming proposition is to generalize the obtained scales, i.e. $(\ref{3DEigS})$, when the 3D-Newtonian potential operator is defined over an arbitrary shape domain $\Omega \subset \mathbb{R}^{3}$. 
\begin{proposition}\label{PropoII}
Assume that $\Omega \subset \mathbb{R}^{3}$ such that $\left\vert \Omega \right\vert \sim a^{3}$. Let $\left( \lambda_{n}(\Omega); e_{n} \right)_{n \geq 0}$ the eigen-system associated to the 3D-Newtonian potential operator $N_{\Omega}(\cdot)$, defined by $(\ref{defN})$. Then, for $n \geq 0$, we have the following behaviour
\begin{equation*}
\lambda_{n}(\Omega) \, \sim \, a^{2} \;\; \text{and} \;\; \left\vert \int_{\Omega} e_{n}(x) \, dx \right\vert \, \lesssim \, a^{\frac{3}{2}}.
\end{equation*}
\end{proposition}
\begin{proof}{of \textbf{Proposition \ref{PropoII}}}
\newline
By definition we have 
\begin{equation}\label{OmegaEigSys}
N_{\Omega}\left( e_{n} \right) = \lambda_{n}(\Omega) \; e_{n},
\end{equation}
and, thanks to $(\ref{KernelScale})$, we get after scaling from $\Omega$ to $\Omega^{\star}$,  
\begin{equation*}
a^{2} \;  N_{\Omega^{\star}}\left( \tilde{e}_{n} \right) = \lambda_{n}(\Omega) \; \tilde{e}_{n},
\end{equation*}
where $\Omega^{\star}$ is such that $\left\vert \Omega^{\star} \right\vert \, \sim \, 1$ and $\Omega = z + a \, \Omega^{\star}$. We see that $\tilde{e}_{n}$ are eigenfunctions of $N_{\Omega^{\star}}(\cdot)$, in the domain $\Omega^{\star}$, related to the eigenvalues $\lambda_{n}(\Omega^{\star}) = a^{-2} \, \lambda_{n}(\Omega)$, where obviously $\lambda_{n}(\Omega^{\star}) \, \sim \, 1$, with respect to the parameter $a$. This indicate, 
\begin{equation}\label{EstEigOmega3D}
\lambda_{n}(\Omega) \, \sim \, a^{2}, \qquad \forall \, n \geq 0.
\end{equation}
Furthermore, by integrating both sides of  $(\ref{OmegaEigSys})$ and taking the modulus in both sides, we obtain  
\begin{equation*}
\lambda_{n}(\Omega) \; \left\vert \int_{\Omega} e_{n}(x) \, dx \right\vert  = \left\vert \int_{\Omega} N_{\Omega}\left( e_{n} \right)(x) \, dx \right\vert \leq \left\Vert 1 \right\Vert_{\mathbb{L}^{2}(\Omega)} \; \left\Vert N_{\Omega}(\cdot) \right\Vert_{\mathcal{L}\left(\mathbb{L}^{2}(\Omega) ; \mathbb{L}^{2}(\Omega) \right)} \; \left\Vert e_{n} \right\Vert_{\mathbb{L}^{2}(\Omega)}.
\end{equation*}
It is know from the spectral theory that $\left\Vert N_{\Omega}(\cdot) \right\Vert_{\mathcal{L}\left(\mathbb{L}^{2}(\Omega) ; \mathbb{L}^{2}(\Omega) \right)} \leq \underset{n}{Sup}\left(\lambda_{n}(\Omega) \right)=\lambda_{0}(\Omega)$, where the last equality is a consequence of the fact that the sequence of eigenvalues is decreasing. In addition, we know that the sequence $\left\{ e_{n}(\cdot) \right\}_{n \geq 0}$ is orthonormalized in $\mathbb{L}^{2}(\Omega)$. Then, 
\begin{equation*}
 \left\vert \int_{\Omega} e_{n}(x) \, dx \right\vert   \leq \left\vert \Omega \right\vert^{\frac{1}{2}} \; \frac{\lambda_{0}(\Omega)}{\lambda_{n}(\Omega)} \; \overset{(\ref{EstEigOmega3D})}{\sim} \; \left\vert \Omega \right\vert^{\frac{1}{2}} \; \sim  \; a^{\frac{3}{2}}.
\end{equation*}
This concludes the proof of Proposition \ref{PropoII}.
\end{proof}
\begin{remark}
As we can see in the proof of Proposition \ref{PropoII}, the \textit{property of domain monotonicity} for the Newtonian potential operator is no longer used. Consequently, for the 3D-Newtonian potential operator, the  \textbf{Hypotheses \ref{Hyp}} is no longer needed. 
\end{remark}

\bigskip



\section{Proof of Theorem \ref{PrincipleResult}}\label{SectionII}
We split the proof into two steps, in the first one we justify the result in the case of a disc and in the second step we prove the result for a general shape satisfying \textbf{Hypotheses \ref{Hyp}}. 
\begin{enumerate}
\item[]
\item The case of a disc $D$ of radius $a$. 
\medskip
\newline
We recall, from \cite[Theorem 4.1]{kalmenov2011boundary}, that the eigenvalues  of the logarithmic potential operator for a disc are given by: 
\begin{equation}\label{EigValDef}
\lambda_{k,j} = a^{2} \, \left( \mu_{j}^{(k)} \right)^{-2}, \qquad k=0,1,2,\cdots \quad \text{and} \quad j=1,2,\cdots 
\end{equation} 
and the corresponding eigenfunctions given by
\begin{equation}\label{Eigfcts}
u_{k,j}(r,\phi) = \begin{LARGE}
\textbf{J}_{k}
\end{LARGE}\left(\mu_{j}^{(k)} \, \frac{r}{a}\right) \, e^{i k \phi}, \qquad r \in [0,a] \quad \text{and} \quad \phi \in [0,2\pi],
\end{equation} 
where $\begin{LARGE}
\textbf{J}_{k}
\end{LARGE}$ is the Bessel function of the first kind of order $k$ and $\mu_{j}^{(k)}$ are the roots of the following transcendental equation
\begin{equation}\label{EigValk>1}
k \, \begin{LARGE}
\textbf{J}_{k}
\end{LARGE}\left( \mu_{j}^{(k)} \right) + \frac{\mu_{j}^{(k)}}{2} \, \left( \begin{LARGE}
\textbf{J}_{k-1}
\end{LARGE}\left( \mu_{j}^{(k)} \right) - \begin{LARGE}
\textbf{J}_{k+1}
\end{LARGE}\left( \mu_{j}^{(k)} \right) \right) = 0, \qquad k=1,2,\cdots,
\end{equation}
and, for $k=0$, 
\begin{equation}\label{EquaBessel}
\begin{LARGE}
\textbf{J}_{0}
\end{LARGE}
\left( \mu_{j}^{(0)} \right) + 2 \, \log(a) \, \mu_{j}^{(0)} \, \begin{LARGE}
\textbf{J}_{1}
\end{LARGE}
\left( \mu_{j}^{(0)} \right) = 0.
\end{equation}
In \cite[Appendix IV, Table III]{carslaw1947conduction}, for some specific values of the parameter $a$, the authors  computed the first six roots of the transcendental equation $(\ref{EquaBessel})$. It is clear, from $(\ref{EquaBessel})$, that the solution $\left\{ \mu_{j}^{(0)} \right\}_{j \geq 1}$ will be dependent on the parameter $a$. To see closely how the solutions depends on $a$, we start   by plotting the graph\footnote{These graphs were produced with Mathematica.} associated to the function, with parameter $a$, defined by 
\begin{equation}\label{DefPsi}
\Psi_{a}(x) := \begin{LARGE}
\textbf{J}_{0}
\end{LARGE}
\left( x \right) + 2 \, \log(a) \, x \, \begin{LARGE}
\textbf{J}_{1}
\end{LARGE}
\left( x \right).
\end{equation}
\begin{figure}[H]
\includegraphics[scale=0.7]{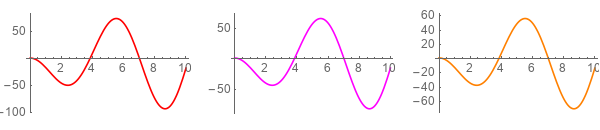}
\caption{A schematic representation of the function $\Psi_{a}(\cdot)$, defined by $(\ref{DefPsi})$. The red graph is associated to  $x \rightarrow \begin{LARGE}
\textbf{J}_{0}
\end{LARGE}
\left( x \right) - 40 \, x \, \begin{LARGE}
\textbf{J}_{1}
\end{LARGE}
\left( x \right)$. The magenta graph is associated to  $x \rightarrow \begin{LARGE}
\textbf{J}_{0}
\end{LARGE}
\left( x \right) - 35 \, x \, \begin{LARGE}
\textbf{J}_{1}
\end{LARGE}
\left( x \right)$. The orange graph is associated to  $x \rightarrow \begin{LARGE}
\textbf{J}_{0}
\end{LARGE}
\left( x \right) - 30 \, x \, \begin{LARGE}
\textbf{J}_{1}
\end{LARGE}
\left( x \right)$.}
\end{figure} 
From the graphs, we can see clearly that the first root, i.e $\mu_{1}^{(0)}$, is small and the other roots $\mu_{j}^{(0)}$, for $j \geq 1$, are moderate. To determine the order of smallness of $\mu_{1}^{(0)}$, we use the asymptotic behavior of $\begin{LARGE}
\textbf{J}_{0}
\end{LARGE}$  and $\begin{LARGE}
\textbf{J}_{1}
\end{LARGE}$ for small argument. More precisely, for $0 < x \ll \sqrt{n+1}$, see \cite[Equation 25]{LANDAU}, we have
\begin{equation}\label{BesselNearZero}
\begin{LARGE}
\textbf{J}_{n}
\end{LARGE}\left( x \right) \sim \frac{1}{n!} \, \left( \frac{x}{2} \right)^{n}. 
\end{equation}
Now, using $(\ref{BesselNearZero})$ we approximate $(\ref{EquaBessel})$ as   
\begin{equation*}
1 +  \log(a) \, \left( \mu_{1}^{(0)} \right)^{2}=0,
\end{equation*}
and this implies that 
\begin{equation*}
\mu_{1}^{(0)} \; \sim \; \frac{1}{\sqrt{\left\vert \log(a) \right\vert}}.
\end{equation*}
Consequently, 
\begin{equation*}
\lambda_{0,1} \; \sim \; a^{2} \, \left\vert \log(a) \right\vert.
\end{equation*}
To study the other roots, we start by
setting $\alpha_{k,j}$, for $k=0,1$ and $j \in \mathbb{N}$, to be the root of order $j$ associated to the Bessel function  $\begin{LARGE}
\textbf{J}_{k}
\end{LARGE}$. Then, thanks to Dixon's theorem, see \cite{watson}, we  know that $\mu_{j}^{(0)}$ will be interlaced between the roots of $\begin{LARGE}
\textbf{J}_{0}
\end{LARGE}$  and the roots of $\begin{LARGE}
\textbf{J}_{1}
\end{LARGE}$. More precisely, 
\begin{equation}\label{DixonResults}
\mu_{1}^{(0)} < \alpha_{0,1} < \alpha_{1,1} \quad \text{and} \quad \alpha_{1,j-1} < \mu_{j}^{(0)} < \alpha_{0,j}, \quad j \geq 2. 
\end{equation}
Using the previous relation and knowing that 
 $ \alpha_{1,1} = 3.8317$, we deduce that $\mu_{j}^{(0)} > 3.8317$, for $j \geq 2$. Therefore, for $j \geq 2$, we have $\mu_{j}^{(0)} \; \sim \; 1$. Hence, from $(\ref{EigValDef})$, we obtain
\begin{equation*}
\lambda_{0,j} \; \sim \; a^{2}, \quad \text{for} \quad j \geq 2.
\end{equation*}
Next, contrary to $(\ref{EquaBessel})$, because the equation $(\ref{EigValk>1})$ is independent on the parameter $a$, we deduce that $\mu_{j}^{(k)} \sim 1$. Hence, the eigenvalues $\lambda_{k,j}$ defined by $(\ref{EigValDef})$ behave as  
\begin{equation*}
\lambda_{k,j} \; \sim \; a^{2}, \quad \text{for} \quad k \geq 1 \quad \text{and} \quad j \geq 1. 
\end{equation*} 
Arranging the obtained results we get: 
\begin{equation}\label{BehaviourEigVal}
\lambda_{0,1} \; \sim \; a^{2} \, \left\vert \log(a) \right\vert  \quad \text{and} \quad \lambda_{k,j} \; \sim \; a^{2}, \quad \text{for} \quad (k,j) \neq (0,1).
\end{equation}
\medskip
\newline
In what follows, when we talk about the first eigenvalue of the Newtonian potential operator in the disc we refer to $\lambda_{0,1}$. This is because the eigenvalues of the Newtonian operator are decreasing and, as proved by $(\ref{BehaviourEigVal})$, $\lambda_{0,1}$ is the highest one.
\medskip
\newline
Similarly to $(\ref{BehaviourEigVal})$, we derive an analogous  result for the integrals of the associated  eigenfunctions $u_{k,j}(\cdot,\cdot)$, defined by $(\ref{Eigfcts})$. We start with the coming computations. 
\begin{eqnarray*}
\int_{D} u_{k,j}(x) \, dx = \int_{0}^{2\pi} \int_{0}^{a}  u_{k,j}(r,\phi) \, r \ dr \, d\phi &=& \int_{0}^{2\pi} \int_{0}^{a}  \begin{LARGE}
\textbf{J}_{k}
\end{LARGE}\left(\mu_{j}^{(k)} \, \frac{r}{a}\right) \, e^{i k \phi} \, r \ dr \, d\phi \\
&=& 2 \, \pi \, \int_{0}^{a}  \begin{LARGE}
\textbf{J}_{0}
\end{LARGE}\left(\mu_{j}^{(0)} \, \frac{r}{a}\right)  \, r \ dr \,\, \delta_{0,k} \\
&=& 2 \, \pi \, \left( \frac{a}{\mu_{j}^{(0)}} \right)^{2} \, \int_{0}^{\mu_{j}^{(0)}}  \begin{LARGE}
\textbf{J}_{0}
\end{LARGE}\left( r \right)  \, r \ dr \,\, \delta_{0,k},
\end{eqnarray*}
where $\delta_{\cdot,\cdot}$ is the Kronecker-symbol. Thanks to \cite[Formula (5), page 206]{carslaw1947conduction}, we know that
\begin{equation*}
x \, \begin{LARGE}
\textbf{J}_{1}
\end{LARGE}\left( x \right) = \int_{0}^{x} r \, \begin{LARGE}
\textbf{J}_{0}
\end{LARGE}\left( r \right) \, dr.
\end{equation*}
Then, 
\begin{equation}\label{IntEigFcts}
\int_{D} u_{k,j}(x) \, dx = 2 \, \pi \,  \frac{a^{2}}{\mu_{j}^{(0)}}  \,  \begin{LARGE}
\textbf{J}_{1}
\end{LARGE}\left( \mu_{j}^{(0)} \right)  \, \delta_{0,k}.
\end{equation}
Later, to define the normalized eigenfunctions, we compute the $\left\Vert u_{0,j} \right\Vert_{\mathbb{L}^{2}(D)}$. We have, 
\begin{eqnarray*}
\left\Vert u_{0,j} \right\Vert^{2}_{\mathbb{L}^{2}(D)}  :=  \int_{D} \left\vert u_{0,j} \right\vert^{2}(x) \, dx = \int_{0}^{2 \, \pi} \, \int_{0}^{a} \left\vert u_{0,j} \right\vert^{2}(r, \phi) \, r \, dr \, d\phi & = & 2 \, \pi \, \int_{0}^{a} \left\vert \begin{LARGE}
\textbf{J}_{0}
\end{LARGE}\left(\mu_{j}^{(0)} \, \frac{r}{a}\right) \right\vert^{2} \, r \, dr \\
&=& 2 \, \pi \, \left( \frac{a}{\mu_{j}^{(0)}} \right)^{2} \, \int_{0}^{\mu_{j}^{(0)}} \left\vert \begin{LARGE}
\textbf{J}_{0}
\end{LARGE}\left(r \right) \right\vert^{2} \, r \, dr.
\end{eqnarray*}
Combining the previous formula with $(\ref{IntEigFcts})$, gives us a formula for the integral of the normalized eigenfunctions. More precisely, 
\begin{equation}\label{IntNEigFcts}
\int_{D} v_{0,j}(x) \, dx := \frac{\int_{D} u_{0,j}(x) \, dx}{\left\Vert u_{0,j} \right\Vert_{\mathbb{L}^{2}(D)}} = \dfrac{\sqrt{2 \, \pi} \; a \; \begin{LARGE}
\textbf{J}_{1}
\end{LARGE}\left( \mu_{j}^{(0)} \right)}{\left[\int_{0}^{\mu_{j}^{(0)}} \left\vert \begin{LARGE}
\textbf{J}_{0}
\end{LARGE}\left(r \right) \right\vert^{2} \, r \, dr\right]^{\frac{1}{2}}}.
\end{equation} 
As a result of the different behaviors of $\left\{ \mu_{j}^{(0)} \right\}_{j \geq 1}$, with respect to the parameter $a$, to estimate $(\ref{IntNEigFcts})$ we split our computations into two steps.
\begin{enumerate}
\item[]
\item For $j=1$, as $\mu_{1}^{(0)}$ is small, we have  
\begin{equation*}
\int_{D} v_{0,1}(x) \, dx  = \frac{\sqrt{2 \, \pi} \; a \; \begin{LARGE}
\textbf{J}_{1}
\end{LARGE}\left( \mu_{1}^{(0)} \right)}{\left[\int_{0}^{\mu_{1}^{(0)}} \left\vert \begin{LARGE}
\textbf{J}_{0}
\end{LARGE}\left(r \right) \right\vert^{2} \, r \, dr\right]^{\frac{1}{2}}} \overset{(\ref{BesselNearZero})}{\sim}  \frac{\sqrt{2 \, \pi} \; a \; \dfrac{\mu_{1}^{(0)}}{2}}{\left[\int_{0}^{\mu_{1}^{(0)}} r \, dr\right]^{\frac{1}{2}}} = \sqrt{\pi} \, a \; \sim \; a .
\end{equation*} 
\item[]
\item For $j \geq 2$, as $\mu_{j}^{(0)}$ is moderate, we have 
\begin{equation*}
\int_{D} v_{0,j}(x) \, dx  = \dfrac{\sqrt{2 \, \pi} \; a \; \begin{LARGE}
\textbf{J}_{1}
\end{LARGE}\left( \mu_{j}^{(0)} \right)}{\left[\int_{0}^{\mu_{j}^{(0)}} \left\vert \begin{LARGE}
\textbf{J}_{0}
\end{LARGE}\left(r \right) \right\vert^{2} \, r \, dr\right]^{\frac{1}{2}}}.
\end{equation*}
By induction on the formula $(\ref{DixonResults})$, we obtain for $j \geq 2$ the following relation 
\begin{equation*}
\alpha_{0,j-1} < \mu_{j}^{(0)} < \alpha_{0,j}.
\end{equation*}
Then, 
\begin{equation*}
\dfrac{\sqrt{2 \, \pi} \; a \; \left\vert \begin{LARGE}
\textbf{J}_{1}
\end{LARGE}\left( \mu_{j}^{(0)} \right) \right\vert}{\left[\int_{0}^{\alpha_{0,j}} \left\vert \begin{LARGE}
\textbf{J}_{0}
\end{LARGE}\left(r \right) \right\vert^{2} \, r \, dr\right]^{\frac{1}{2}}} < \left\vert \int_{D} v_{0,j}(x) \, dx \right\vert  < \dfrac{\sqrt{2 \, \pi} \; a \; \left\vert \begin{LARGE}
\textbf{J}_{1}
\end{LARGE}\left( \mu_{j}^{(0)} \right) \right\vert}{\left[\int_{0}^{\alpha_{0,j-1}} \left\vert \begin{LARGE}
\textbf{J}_{0}
\end{LARGE}\left(r \right) \right\vert^{2} \, r \, dr\right]^{\frac{1}{2}}},
\end{equation*}
or, equivalently, 
\begin{equation*}
\dfrac{\sqrt{2 \, \pi} \; a \; \left\vert \begin{LARGE}
\textbf{J}_{1}
\end{LARGE}\left( \mu_{j}^{(0)} \right) \right\vert}{\alpha_{0,j} \, \left[\int_{0}^{1} \left\vert \begin{LARGE}
\textbf{J}_{0}
\end{LARGE}\left(\alpha_{0,j} \, r \right) \right\vert^{2} \, r \, dr\right]^{\frac{1}{2}}} < \left\vert \int_{D} v_{0,j}(x) \, dx \right\vert  < \dfrac{\sqrt{2 \, \pi} \; a \; \left\vert \begin{LARGE}
\textbf{J}_{1}
\end{LARGE}\left( \mu_{j}^{(0)} \right) \right\vert}{\alpha_{0,j-1} \, \left[\int_{0}^{1} \left\vert \begin{LARGE}
\textbf{J}_{0}
\end{LARGE}\left(\alpha_{0,j-1} \, r \right) \right\vert^{2} \, r \, dr\right]^{\frac{1}{2}}}.
\end{equation*}
Now, thanks to \cite[Formula (2), page 205]{carslaw1947conduction}, we know that
\begin{equation*}
\int_{0}^{1} \left( \begin{LARGE}
\textbf{J}_{0}
\end{LARGE}\left(\alpha_{0,k} \, r \right) \right)^{2} \, r \, dr = \frac{1}{2} \, \left( \begin{LARGE}
\textbf{J}_{1}
\end{LARGE}\left(\alpha_{0,k} \right) \right)^{2} \sim 1, \quad \text{for} \quad k \in \mathbb{N},
\end{equation*}
hence,
\begin{equation*}
\dfrac{2 \, \sqrt{ \pi} \; a \; \left\vert \begin{LARGE}
\textbf{J}_{1}
\end{LARGE}\left( \mu_{j}^{(0)} \right) \right\vert}{\alpha_{0,j} \, \left\vert \begin{LARGE}
\textbf{J}_{1}
\end{LARGE}\left(\alpha_{0,j} \right) \right\vert} < \left\vert \int_{D} v_{0,j}(x) \, dx \right\vert  < \dfrac{2 \, \sqrt{ \pi} \; a \; \left\vert \begin{LARGE}
\textbf{J}_{1}
\end{LARGE}\left( \mu_{j}^{(0)} \right) \right\vert}{\alpha_{0,j-1} \, \left\vert \begin{LARGE}
\textbf{J}_{1}
\end{LARGE}\left(\alpha_{0,j-1} \right) \right\vert}.
\end{equation*}
As $\alpha_{0,k} \sim 1$ and $\left\vert \begin{LARGE}
\textbf{J}_{1}
\end{LARGE}\left(\alpha_{0,k} \right) \right\vert \sim 1$, for $k=j-1$ or $k=j$, with respect to the parameter $a$, we deduce that
\begin{equation*}
 a \; \left\vert \begin{LARGE}
\textbf{J}_{1}
\end{LARGE}\left( \mu_{j}^{(0)} \right) \right\vert \lesssim \left\vert \int_{D} v_{0,j}(x) \, dx \right\vert \lesssim a \; \left\vert \begin{LARGE}
\textbf{J}_{1}
\end{LARGE}\left( \mu_{j}^{(0)} \right) \right\vert, 
\end{equation*}
or, equivalently, 
\begin{equation*}
\left\vert \int_{D} v_{0,j}(x) \, dx \right\vert \; \sim \; a \; \left\vert \begin{LARGE}
\textbf{J}_{1}
\end{LARGE}\left( \mu_{j}^{(0)} \right) \right\vert \overset{(\ref{EquaBessel})}{=} \frac{a \; \left\vert \begin{LARGE}
\textbf{J}_{0}
\end{LARGE}\left( \mu_{j}^{(0)} \right) \right\vert}{2 \, \mu_{j}^{(0)} \, \left\vert \log(a) \right\vert}.
\end{equation*}
The final step consist in using the fact that\footnote{This can be proved using the integral representation of the Bessel function 
\begin{equation*}
\begin{LARGE}
\textbf{J}_{0}
\end{LARGE}\left( x \right) := \frac{1}{\pi} \, \int_{0}^{\pi} \cos\left( x \, \sin(\tau)\right) \, d\tau, 
\end{equation*}
see \cite[Formula 9.19, page 230]{Temme}.}  $\left\vert \begin{LARGE}
\textbf{J}_{0}
\end{LARGE}\left( x \right) \right\vert \leq 1, \, \forall \, x \in \mathbb{R}$, see \cite[Formula (5), page 31]{watson}. Hence,  
\begin{equation*}
\left\vert \int_{D} v_{0,j}(x) \, dx \right\vert \; \sim \; a \;  \left\vert \log(a) \right\vert^{-1}.
\end{equation*}
\end{enumerate} 
Analogously to $(\ref{BehaviourEigVal})$, after arranging the obtained results we deduce that: 
\begin{equation}\label{IranProtests}
\int_{D} v_{0,1}(x) \, dx \;\; \sim \;\; a  \quad \text{and} \quad \int_{D} v_{0,j}(x) \, dx \;\; \sim \;\; a \, \left\vert \log(a) \right\vert^{-1}, \quad j \geq 2.
\end{equation}
Obviously, 
\begin{equation*}
\int_{D} v_{k,j}(x) \, dx = 0, \quad k \geq 1 \quad \text{and}  \quad j \geq 1.
\end{equation*}
\item[]
\item The case of an arbitrary shape $\Omega$.\medskip
\newline
To estimate the behavior of the eigenvalues of the Newtonian potential operator defined over an arbitrary domain $\Omega$, we proceed in two steps. 
\begin{enumerate}
\item[]
\item \label{PDM}
 We start by proving the \textit{ property of domain monotonicity} for the Newtonian potential operator: eigenvalues of the Newtonian potential operator monotonously increase when the domain is enlarged, i.e., $\lambda_{n}\left( \Omega_{1} \right) \leq \lambda_{n}\left( \Omega_{2} \right)$ if $ \Omega_{1} \subset \Omega_{2}$. To accomplish this, we define an extension operator $\bm{P}$, as follows: 
\smallskip
\newline
\begin{eqnarray}\label{Extension}
\nonumber
\bm{P} := \mathbb{L}^{2}\left( \Omega_{1} \right) & \longrightarrow & \mathbb{L}^{2}\left( \Omega_{2} \right) \\
f(\cdot) & \longrightarrow & \bm{P}\left( f \right)(\cdot) := \tilde{f}(\cdot) := f(\cdot) \; \underset{\Omega_{1}}{\chi}(\cdot) \; + \; 0 \; \underset{\Omega_{2} \setminus \Omega_{1}}{\chi}(\cdot),
\end{eqnarray}  
where we have assumed that $\Omega_{1} \subset \Omega_{2}$. Manifestly, we have the following injection 
\begin{equation}\label{Injection}
\mathbb{L}^{2}\left( \Omega_{1} \right) \hookrightarrow \mathbb{L}^{2}\left( \Omega_{2} \right).
\end{equation}
In similar manner we define the extension of a function represented by the Newtonian operator 
\begin{equation*}
u(x) = \int_{\Omega_{1}} \Phi_{0}(x,y) \, f(y) \, dy, \quad x \in \Omega_{1},   
\end{equation*} 
to the domain $\Omega_{2}$,  as follows
\begin{equation*}
\tilde{u}(x) = \int_{\Omega_{2}} \Phi_{0}(x,y) \, \tilde{f}(y) \, dy, \quad x \in \Omega_{2},   
\end{equation*}
where $\tilde{f}(\cdot)$ is the extension of $  f(\cdot)$, with zero in $\Omega_{2} \setminus \Omega_{1}$, defined by $(\ref{Extension})$.
Now, from the min-max principle applied to the Newtonian potential operator, see  \cite[Theorem 4, page 318-319]{lax2002functional}, we have\footnote{In $(\ref{min-max})$, the positivity of the operator $N_{\Omega_{1}}(\cdot)$ ensures that all eigenvalues are non-negative.} 
\begin{equation}\label{min-max}
\lambda_{k}\left( \Omega_{1} \right) = \underset{\Xi_{k}}{Sup} \quad \underset{u \in \left( \Xi_{k} \cap  \mathbb{L}^{2}(\Omega_{1}) \right)}{Inf} \quad \dfrac{\int_{\Omega_{1}} N_{\Omega_{1}}\left(u\right)(x) u(x) \; dx}{\int_{\Omega_{1}} \left\vert u \right\vert^{2}(x) \; dx},
\end{equation}
where $\Xi_{k} \subset \mathbb{L}^{2}(\mathbb{R}^{2})$ is $k$ dimensional subspace.  Now, using the injection $(\ref{Injection})$ we deduce that:
\begin{equation}\label{Black}
\lambda_{k}\left( \Omega_{1} \right) \leq  \underset{\Xi_{k}}{Sup} \quad \underset{u \in \left( \Xi_{k} \cap   \mathbb{L}^{2}(\Omega_{2}) \right)}{Inf} \quad \dfrac{\int_{\Omega_{1}} N_{\Omega_{1}}\left(u\right)(x) u(x) \; dx}{\int_{\Omega_{1}} \left\vert u \right\vert^{2}(x) \; dx}.
\end{equation}
Exploiting the properties of the extension function, defined by $(\ref{Extension})$, we obtain 
\begin{equation*}
\int_{\Omega_{1}} N_{\Omega_{1}}\left(u\right)(x) u(x) \; dx = \int_{\Omega_{2}} N_{\Omega_{2}}\left(\tilde{u}\right)(x) \tilde{u}(x) \; dx \quad \text{and} \quad \int_{\Omega_{1}} \left\vert u \right\vert^{2}(x) \; dx = \int_{\Omega_{2}} \left\vert \tilde{u} \right\vert^{2}(x) \; dx.
\end{equation*}
Hence, $(\ref{Black})$ becomes, 
\begin{equation*}
\lambda_{k}\left( \Omega_{1} \right) \leq  \underset{\Xi_{k}}{Sup} \quad \underset{\tilde{u} \in \left( \Xi_{k} \cap \mathbb{L}^{2}(\Omega_{2}) \right)}{Inf} \quad \dfrac{\int_{\Omega_{2}} N_{\Omega_{2}}\left(\tilde{u}\right)(x) \tilde{u}(x) \; dx}{\int_{\Omega_{2}} \left\vert \tilde{u} \right\vert^{2}(x) \; dx} = \lambda_{k}\left( \Omega_{2} \right).
\end{equation*}
Finally, 
\begin{equation}\label{Monotonicity}
\text{if} \quad \Omega_{1} \subset \Omega_{2} \quad \text{we obtain} \quad \lambda_{k}\left( \Omega_{1} \right) \leq   \lambda_{k}\left( \Omega_{2} \right),
\end{equation}
and this ends the proof of the monotonicity property for the eigenvalues of the Newtonian potential operator. 
\item[]
\item At this stage, we use the proved monotonicity property, given by $(\ref{Monotonicity})$, by assuming the existence of two discs $B_{1} := B(z_{1},\rho_{1})$ and $B_{2} := B(z_{2},\rho_{2})$, where $z_{j} \in \Omega$ and  $\rho_{j} > 0$, for $j=1,2$, such that: 
\begin{equation*}
B_{1} \subset \Omega \subset B_{2} \quad \text{and} \quad \left\vert B_{j} \right\vert \sim a^{2}, \;\; j=1,2.
\end{equation*}
Thanks to $(\ref{Monotonicity})$ we derive, from the previous formula, the following relation, 
\begin{equation}\label{B1OmegaB2}
\lambda_{k}\left( B_{1} \right) \leq \lambda_{k}\left( \Omega \right) \leq \lambda_{k}\left( B_{2} \right).
\end{equation} 
\medskip
\newline
Now, by combining $(\ref{B1OmegaB2})$ with $(\ref{BehaviourEigVal})$, we derive the following behavior of the eigenvalues of the Newtonian potential operator, defined over an arbitrary shape $\Omega$.
\begin{equation}\label{EstimationEigOmega}
\lambda_{0}\left( \Omega \right) \, \sim \, a^{2} \, \left\vert \log(a) \right\vert  \quad \text{and} \quad \lambda_{k}\left( \Omega \right) \, \sim \, a^{2}, \quad \text{for} \quad k \geq 1.
\end{equation}
\end{enumerate}
\medskip
Regarding the estimation of the integral of the eigenfunction $f_{n}$, of the Newtonian potential operator defined over $\Omega$, i.e. 
\begin{equation}\label{AddedEqua}
N_{\Omega}\left( f_{n} \right) = \lambda_{n}(\Omega) \; f_{n},\, \quad \text{in} \quad \Omega,
\end{equation}
and as suggested by the behavior of the eigenvalues $\left\{ \lambda_{n}\left( \Omega \right) \right\}_{n \geq 0}$, i.e. $(\ref{EstimationEigOmega})$, we split the study into two cases.
\begin{enumerate}
\item[]
\item[i)] \label{EstimationFirstEigFct}
 For $n = 0$, we know from $(\ref{EstimationEigOmega})$, that $\lambda_{0}(\Omega) = \beta_{0} \, a^{2}\, \left\vert \log(a) \right\vert$, where $\beta_{0}$ is a positive constant independent on the parameter $a$. Moreover, after scaling the equation $(\ref{AddedEqua})$, from $\Omega$ to $\Omega^{\star}$, and taking the inner product with respect to $\tilde{f}_{n}$, we end up with 
\begin{equation*}
\left\langle \tilde{f}_{n}; N_{\Omega^{\star}}\left( \tilde{f}_{n} \right) \right\rangle_{\mathbb{L}^{2}(\Omega^{\star})} = \frac{\lambda_{n}(\Omega)}{a^{2}} \; \left\Vert \tilde{f}_{n} \right\Vert_{\mathbb{L}^{2}(\Omega^{\star})} + \frac{1}{2 \, \pi} \, \log(a) \, \left( \int_{\Omega^{\star}} \tilde{f}_{n} \right)^{2}. 
\end{equation*}
Next, we set $\overline{f}_{n} := \dfrac{ \tilde{f}_{n}}{\left\Vert \tilde{f}_{n} \right\Vert_{\mathbb{L}^{2}\left( \Omega^{\star} \right)}}$ and $\tilde{\lambda}_{n} := \left\langle \tilde{f}_{n}; N_{\Omega^{\star}}\left( \tilde{f}_{n} \right) \right\rangle_{\mathbb{L}^{2}(\Omega^{\star})}$ we obtain from the previous equation
\begin{equation}\label{Formulafrom2DWork}
\tilde{\lambda}_{n} = \frac{\lambda_{n}(\Omega)}{a^{2}} + \frac{1}{2 \, \pi} \, \log(a) \, \left( \int_{\Omega^{\star}} \overline{f}_{n}(x) \, dx \right)^{2},
\end{equation}
where $\Omega^{\star} = \dfrac{-z+\Omega}{a}, \, \left\vert \Omega^{\star} \right\vert \, \sim \, 1 $ and $\left\{ \tilde{\lambda}_{n} \right\}_{n \geq 0}$ is a positive and uniformly bounded sequence, i.e. $0 < \tilde{\lambda}_{n} \leq \underset{n}{Sup} \; \tilde{\lambda}_{n} < C^{te}$, where $C^{te}$ is a uniformly constant. Using the scale of $\lambda_{0}(\Omega)$ and the previous formula give us 
\begin{equation}\label{Equa1}
\tilde{\lambda}_{0} = \frac{\beta_{0} \, a^{2} \, \left\vert \log(a) \right\vert}{a^{2}} + \frac{1}{2 \, \pi} \, \log(a) \, \left( \int_{\Omega^{\star}} \overline{f}_{0}(x) \, dx \right)^{2},
\end{equation}
and this implies 
\begin{equation}\label{Equa2}
\left\vert \int_{\Omega^{\star}} \overline{f}_{0}(x) \, dx \right\vert = \sqrt{2 \, \pi \, \left( \beta_{0} - \frac{\tilde{\lambda}_{0}}{\left\vert \log(a) \right\vert} \right)} \;\;\sim \;\; 1.
\end{equation} 
\item[]
\item[ii)] For $n \geq 1$, we know from $(\ref{EstimationEigOmega})$, that $\lambda_{n}(\Omega) = \beta_{n} \, a^{2}\, $, where $\beta_{n}$ is a positive constant independent on the parameter $a$. Hence, using $(\ref{Formulafrom2DWork})$, we obtain
\begin{equation*}
\tilde{\lambda}_{n} = \frac{\beta_{n} \, a^{2} }{a^{2}} + \frac{1}{2 \, \pi} \, \log(a) \, \left( \int_{\Omega^{\star}} \overline{f}_{n}(x) \, dx \right)^{2},
\end{equation*}
and, knowing that $\tilde{\lambda}_{n} > 0$,  
\begin{equation*}
\left( \int_{\Omega^{\star}} \overline{f}_{n}(x) \, dx \right)^{2} \leq \beta_{n} \; 2 \, \pi \; \left\vert \log(a) \right\vert^{-1}.
\end{equation*}
This implies, 
\begin{equation*}
\left\vert \int_{\Omega^{\star}} \overline{f}_{n}(x) \, dx \right\vert  \;\; \lesssim \;\; \left\vert \log(a) \right\vert^{-\frac{1}{2}}.
\end{equation*}
Next, we show that the previous obtained estimation can be improved to be of order $\left\vert \log(a) \right\vert^{- 1}$, instead of $\left\vert \log(a) \right\vert^{-\frac{1}{2}}$. Now, after scaling the equation $(\ref{AddedEqua})$ and integrating again the obtained equation we obtain   
\begin{equation}\label{LBII}
\int_{\Omega^{\star}} \overline{f}_{n}(x) \, dx = \frac{\int_{\Omega^{\star}} N_{\Omega^{\star}}\left(1 \right)(x) \overline{f}_{n}(x) \, dx}{\left[  \frac{\lambda_{n}\left( \Omega \right)}{a^{2}} - \frac{1}{2 \, \pi} \, \left\vert \log(a) \right\vert \, \left\vert \Omega^{\star} \right\vert \right]},
\end{equation}
and knowing that $\lambda_{n}\left( \Omega \right) = \beta_{n} \, a^{2}$ we get
\begin{equation*}
\int_{\Omega^{\star}} \overline{f}_{n}(x) \, dx = \frac{\int_{\Omega^{\star}} N_{\Omega^{\star}}\left(1 \right)(x) \overline{f}_{n}(x) \, dx}{\left[ \beta_{n} - \frac{1}{2 \, \pi} \, \left\vert \log(a) \right\vert \, \left\vert \Omega^{\star} \right\vert \right]}.
\end{equation*}
By taking the modulus, in both sides, of the previous relation, using the fact that $\left\Vert N_{\Omega^{\star}} \right\Vert_{\mathcal{L}\left(\mathbb{L}^{2}(\Omega^{\star}) ; \mathbb{L}^{2}(\Omega^{\star})\right)} \; \sim \; 1$ and recalling that $\overline{f}_{n}(\cdot)$ are orthonormalized eigenfunctions in $\mathbb{L}^{2}(\Omega^{\star})$, we deduce  that 
\begin{equation*}
\left\vert \int_{\Omega^{\star}} \overline{f}_{n}(x) \, dx \right\vert \; \lesssim \;  \left\vert \log(a) \right\vert^{-1}. 
\end{equation*}

\item[]
\end{enumerate}

Finally, correspondingly to $(\ref{IranProtests})$ and for an arbitrary shape domain $\Omega$, we obtain after rescaling back to $\Omega$ the following behaviour of the integral of the eigenfunctions of the Newtonian potential operator with respect to the parameter $a$.
\begin{equation}\label{Iran'sWomen}
\int_{\Omega} f_{0}(x) \, dx \;\; \sim \;\; a \quad \text{and} \quad  \vert \int_{\Omega} f_{n}(x) \, dx \vert \;\; \lesssim \;\; a \, \left\vert \log(a) \right\vert^{-1}, \quad \text{for} \;\; n \geq 1.
\end{equation}
\end{enumerate}





\end{document}